\documentclass[a4paper,10pt]{article}
	\usepackage[utf8]{inputenc}
	\usepackage{amssymb}
	\usepackage{amsmath}
	\usepackage{amsfonts}
	\usepackage{amsthm}
	\usepackage{stmaryrd}
	\usepackage{mathrsfs}
	\usepackage{algorithm,algorithmic,refcount}
	\usepackage{xcolor}
	\usepackage{tkz-euclide}

	\usepackage{amsopn,url}
	\usepackage{epstopdf}
	
	\usepackage{tikz}
	\usetikzlibrary{matrix,decorations.pathreplacing,calc}

	\usepackage{natbib}
	\setcitestyle{square}
	\setcitestyle{numbers}
	\setcitestyle{comma}
	
	\newcommand{\R}{\mathbb {R}}
	\newcommand{\I}{{I}}
	\newcommand{\J}{{J}}
	\newtheorem{theorem}{Theorem}
	\newtheorem{example}[theorem]{Example}
	\newtheorem{corollary}[theorem]{Corollary}
	\newtheorem{lemma}[theorem]{Lemma}
	\newtheorem{remark}[theorem]{Remark}

	\newcommand{\newCol}{\ensuremath{j_t}}
    \newcommand{\newRow}{\ensuremath{i_t}}
    \newcommand{\minRatio}{\ensuremath{\,\mathrm{minRatio}}}
    \newcommand{\Vol}{\ensuremath{\,\mathrm{Vol}}}
    \newcommand{\rank}{\ensuremath{\,\mathrm{rank}}}
    \newcommand{\upperBound}{\ensuremath{\,\mathrm{bound}}}

    \author{Alice Cortinovis\footnote{Institute of Mathematics, EPF Lausanne, 1015 Lausanne, Switzerland. E-mail: alice.cortinovis@epfl.ch. The work of Alice Cortinovis has been supported by the SNSF research project \emph{Fast algorithms from low-rank updates}, grant number: 200020\_178806.} \and Daniel Kressner\footnote{Institute of Mathematics, EPF Lausanne, 1015 Lausanne, Switzerland. E-mail: daniel.kressner@epfl.ch.} }
	  
	\title{Low-rank approximation in the Frobenius norm by column and row subset selection}
	
\begin{document}
	
	\maketitle
	
\begin{abstract}
A CUR approximation of a matrix $A$ is a particular type of low-rank approximation $A\approx CUR$, where $C$ and $R$ consist of columns and rows of $A$, respectively. One way to obtain such an approximation is to apply column subset selection to $A$ and $A^T$.
In this work, we describe a numerically robust and much faster variant of the column subset selection algorithm proposed by Deshpande and Rademacher, which guarantees an error close to the best approximation error in the Frobenius norm.
For cross approximation, in which $U$ is required to be the inverse of a submatrix of $A$ described by the intersection of $C$ and $R$, we obtain a new algorithm with an error bound that stays within a factor $k+1$ of the best rank-$k$ approximation error in the Frobenius norm. To the best of our knowledge, this is the first deterministic polynomial-time algorithm for which this factor is bounded by a polynomial in $k$. Our derivation and analysis of the algorithm is based on derandomizing a recent existence result by Zamarashkin and Osinsky. To illustrate the versatility of our new column subset selection algorithm, an extension to low multilinear rank approximations of tensors is provided as well.         
\end{abstract}

\section{Introduction}

Given an $m\times n$ matrix $A$ and an integer $k$, typically much smaller than $m$ and $n$, the \emph{column subset selection problem} aims at determining an index set $I \subset \{1, \ldots, n\}$ of cardinality $k$ such that the corresponding $k$ columns $A(:,I)$ represent a good approximation of the range of $A$. This problem has broad applications in a diversity of disciplines, including scientific computing, model reduction, and statistical data analysis. While column subset selection is a classical problem in numerical linear algebra, closely connected to rank-revealing QR factorizations~\cite{Chan1987,Chandrasekaran1994,Gu1996}, new significant theoretical and algorithmic developments have been achieved during the last two decades within the model reduction and theory of algorithms communities. In particular, this concerns the interplay between column subset selection and interpolation~\cite{Barrault2004,Chaturantabut2010,Drmac2016} as well as the development and analysis of randomized algorithms, see~\cite{Boutsidis2009,Drineas2008,Frieze2004,Woodruff2014} for a few references representing this research direction.

This paper is concerned with algorithmic improvements and extensions of the seminal work by Deshpande, Rademacher and co-authors~\cite{Deshpande2010,Deshpande2006} on column subset selection. In~\cite{Deshpande2006} the existence of an index set $I$ such that
\begin{equation}\label{eq:CS}
    \| A - A(:,I)A(:,I)^+ A \|_F^2 \le (k+1)\left(\sigma_{k+1}(A)^2 + \ldots + \sigma_{\min\{m,n\}}(A)^2\right)
\end{equation}
has been established. Here, $\|\cdot\|_F$ and $(\cdot)^+$ denote the Frobenius norm and the Moore--Penrose inverse of a matrix, respectively. We let $\sigma_{1}(A) \ge \sigma_{2}(A) \ge  \cdots$ denote the singular values of $A$. Note that $A(:,I)A(:,I)^+$ is an orthogonal projector and the bound~\eqref{eq:CS} measures how well all the columns of $A$ are approximated by the subset of column contained in $I$.
The bound~\eqref{eq:CS} is remarkable because the singular value decomposition (SVD) of $A$ implies that the best approximation error $\| A - QQ^+ A \|_F^2$ attained by an \emph{arbitrary} $m\times k$ matrix $Q$ is given by  $\sigma_{k+1}(A)^2 + \ldots + \sigma_{\min\{m,n\}}(A)^2$. The bound~\eqref{eq:CS} is larger by a factor that is only linear in $k$. More generally, we will call any quasi-optimal bound with a factor that is at most polynomial in $k$ (and independent of $m,n$ or $A$) a \emph{polynomial bound}. The proof of~\eqref{eq:CS} proceeds by defining a suitable discrete probability distribution on index tuples such that the expected value of the error with respect to this distribution satisfies the bound. This then implies the existence of at least one index set satisfying the bound as well.
We remark that the factor $k+1$ in~\eqref{eq:CS} cannot be improved~\cite[Proposition 3.3]{Deshpande2006}. In~\cite{Deshpande2010}, a deterministic algorithm has been developed by derandomizing this approach using the method of conditional expectations. These conditional expectations are given in terms of coefficients of certain characteristic polynomials and the algorithm from~\cite{Deshpande2010} attains efficiency by cheaply updating these coefficients. However, it is well known that working with characteristic polynomials in finite precision arithmetic is prone to massive numerical cancellation~\cite{RehmanIpsen2011} and, as we will see, the algorithm from~\cite{Deshpande2010} is also affected by numerical instability. Our first contribution, presented in Section~\ref{sec:colsel}, consists of deriving a formulation of the algorithm that updates singular values instead of coefficients of characteristic polynomials. While our new variant enjoys the same favorable complexity, numerical experiments with matrices of different singular value decay indicate that it is numerically robust, achieving~\eqref{eq:CS} up to the level of roundoff error. Based on a minor extension of the theory from~\cite{Deshpande2010,Deshpande2006}, we will also present a modification of the column selection strategy that results in significant speed ups of the algorithm.

In Section~\ref{sec:matrixapprox}, we extend the developments from~\cite{Deshpande2010} to the problem of determining a rank-$k$ approximation of the form
\[
 A\approx CUR,
\]
where $C = A(:,J)$ and $R = A(I,:)$ contain $k$ selected columns and rows of $A$, respectively. There is a simple and well established strategy to derive such an approximation, see, e.g.,~\cite{Drineas2008,Sorensen2016}: One first applies column subset selection to $A$ and $A^T$ in order to determine $C$ and $R$, respectively. Given $C$ and $R$, the choice 
\begin{equation} \label{eq:cur}
 U = C^+ A R^+
\end{equation}
then minimizes the Frobenius norm error. We will show that this strategy combined with~\eqref{eq:CS} results in an error that is at most a factor $\sqrt{2k+2}$ larger than the best rank-$k$ approximation error. While this is clearly a favorable bound, the choice~\eqref{eq:cur} comes with a disadvantage. It involves the full matrix $A$, but sometimes only partial information on $A$ is available and has been used to determine $I,J$. One example for such a situation is the Chebfun2 construction for approximating bivariate functions~\cite{Townsend2013}, which uses a coarse discretization to cheaply determine $I,J$ and then evaluates the full matrix $A$ only along the cross containing the rows and columns determined by $I$ and $J$, respectively. The choice $U = A(I,J)^{-1}$ then leads to a rank-$k$ approximation of the form
\begin{equation*} 
 A\approx A(:,J) A(I,J)^{-1} A(I,:),
\end{equation*}
which is often called cross approximation. Choosing $I,J$ via column subset selection is not advisable in this setting; it may lead to (nearly) singular $A(I,J)$ and result in an unfavorable approximation error.
 On the other hand, Goreinov and Tyrtyshnikov~\cite{GoreinovTyrtyshnikov2001} have established a polynomial bound for cross approximation in the maximum norm when choosing $I,J$ such that the volume of $A(I,J)$ is maximal. Recently, Zamarashkin and Osinsky~\cite{Zamarashkin2018} derived a polynomial bound in the Frobenius norm by extending the techniques from~\cite{Deshpande2006}. However, as far as we know, there is no polynomial time deterministic algorithm that guarantees a polynomial bound (in any norm); popular greedy algorithms lead to exponential bounds~\cite{CKM2019,Harbrecht2012} at best.
One major contribution of this work is to derive such an algorithm via an extension of~\cite{Deshpande2010}; our algorithm guarantees a Frobenius norm error that is at most a factor $k+1$ larger than the best approximation error.

Section~\ref{sec:tensors} contains an extension to tensors. In particular, we derive a deterministic algorithm that obtains a multilinear low-rank approximation that is constructed from the fibers of the tensor and satisfies a polynomial bound. Although our approach is a relatively straightforward extension of~\eqref{eq:cur} and related approaches have been proposed in the literature~\cite{Drineas2007, Goreinov2008, Saibaba2016}, we are not aware that such an algorithm has been explicitly spelled out and analyzed.

\section{Column subset selection} \label{sec:colsel}


We start by providing more details on the approach from~\cite{Deshpande2010,Deshpande2006} for the column subset selection problem. In the following we consider a matrix $A \in \R^{m \times n}$ with $m \le n$ and rank at least $k$. We let $a_i$ denote the $i$th column of $A$ and $\pi_{i_1, \ldots, i_k}A$ the orthogonal projection of $A$ on the subspace spanned by the columns $a_{i_1}, \ldots, a_{i_k}$, that is,
\begin{equation*}
    \pi_{i_1, \ldots, i_k}A := A(:, I ) \cdot A(:, I)^{+}\cdot A = QQ^T A,
\end{equation*}
where $I = ( i_1,\ldots,i_k ) \in \{1, \ldots, n\}^k$ and $Q$ denotes an orthonormal basis of $A(:, I)$. Let us emphasize that $I$ is now a \emph{tuple}. Although order is not important and we are ultimately interested in an index \emph{set}, working with tuples simplifies the subsequent definition and manipulation of probability distributions.
The \emph{volume} of a rectangular matrix $B \in \R^{m \times k}$ with $k \le m$ is defined as $\Vol(B) := \prod_{i=1}^k \sigma_i(B)^2$. Note that $\Vol(B)^2 = \det(B^TB)$.

We now define a discrete probability distribution on integer tuples of the form $I \in \{1, \ldots, n\}^k$ corresponding to a selection of $k$ columns from $A$. For this purpose, let $X = (X_1, \ldots, X_k)$ be a $k$-tuple of random variables with values in $\{1, \ldots, n\}$ such that
\begin{equation} \label{eq:discpdf}
    \mathbb{P}(X = I) := \frac{\Vol(A(:,I))^2}{\sum_{J \in \{1, \ldots, n\}^k} \Vol(A(:,J))^2}.
\end{equation}
By convention, $\Vol\left(A(:,I)\right) = 0$ whenever $i_1,\ldots, i_k$ contain repeated indices.
Then~\cite[Theorem 1.3]{Deshpande2006} shows that
\begin{equation}\label{eq:ECS}
    \mathbb{E} [ \| A - \pi_{X_1, \ldots, X_k} A \|_F^2] \le (k+1)\left(\sigma_{k+1}^2 + \ldots + \sigma_{m}^2 \right).
\end{equation}
In particular, this implies the existence of $I$ satisfying this bound. 

In view of~\eqref{eq:discpdf} and the prominent role played by maximum volume submatrices in low-rank approximation~\cite{GoreinovTyrtyshnikov2001}, it is tempting to expect that the $k$ columns of maximum volume satisfy~\eqref{eq:CS}. However, not only that it is NP hard to choose such columns~\cite{Civril2009}, but they also fail to have this property. For instance, for $k=1$ consider the $2 \times n$ matrix
\begin{equation*}
    A = \begin{bmatrix}
        a(1+\varepsilon) & b & b & \ldots & b \\
        -b(1 + \varepsilon) & a & a & \ldots & a \\
    \end{bmatrix}
\end{equation*}
with $a^2 + b^2 = 1$ and $\varepsilon > 0$. The column of maximum volume (that is, of maximum Euclidean norm) is the first one. The approximation error obtained by this choice is given by $\| A - \pi_1 A \|_F^2 = n-1$, which is much larger than $2 \sigma^2_2 = 2 (1+\varepsilon)^2$ for $\varepsilon$ sufficiently small. Note that choosing any of the other columns yields the best approximation error $(1 + \varepsilon)^2 = \sigma_2^2$.

\subsection{Algorithm by Deshpande and Rademacher}\label{sec:detAlgorithm}

Deshpande and Rademacher~\cite{Deshpande2010} derived a deterministic algorithm for column subset selection by derandomizing~\eqref{eq:ECS} using the 
method of conditional expectations.

More specifically, the first step of the algorithm chooses an index $i_1$ such that
\[    \mathbb{E} \left [ \| A - \pi_{X_1, \ldots,X_k}A \|_F^2 \mid X_1 = i_1 \right ]
\]
is minimized. By construction, this quantity still satisfies the bound~\eqref{eq:ECS}. More generally, having $t-1$ indices $i_1,\ldots,i_{t-1}$ selected, step $t$ chooses an index $i_t$ such that
\begin{equation}\label{eq:choiceit}
    \mathbb{E} \left [ \| A - \pi_{X_1, \ldots,X_k}A \|_F^2 \mid X_1 = i_1, \ldots, X_{t-1} = i_{t-1}, X_t = i_t \right ]
\end{equation}
is minimized. After $k$ steps we arrive at an index set $I$ of cardinality $k$ such that the desired bound~\eqref{eq:CS} holds.

For the algorithm to be practical, it is crucial to compute the conditional expectations~\eqref{eq:choiceit} efficiently. Lemma 21 in~\cite{Deshpande2010} shows that
\begin{equation*}
    \mathbb{E} \left[ \| A - \pi_{X_1, \ldots, X_k} A \|_F^2 \mid X_1 = i_1, \ldots, X_t = i_t\right ] = (k-t+1)\frac{c_{m-k+t-1}(BB^T)}{c_{m-k+t}(BB^T)},
\end{equation*}
where the right-hand side involves the matrix $B = A - \pi_{i_1,\ldots,i_t} A$ and coefficients $c_j\equiv c_j(BB^T)$ of the characteristic polynomial
\begin{equation}\label{eq:coeffCharPoly}
(-\lambda)^m + c_{m-1}(-\lambda)^{m-1} + \ldots + c_1(-\lambda) + c_0 := \det(BB^T - \lambda I).
\end{equation}
It is therefore required to compute in every step for all values of $i$, the ratios 
\begin{equation}\label{eq:wantedRatios}
\frac{c_{m-k+t-1}(B_iB_i^T)}{c_{m-k+t}(B_iB_i^T)}
\end{equation}
where $B_i = A - \pi_{i_1,\ldots,i_{t-1},i} A$. 

In the following, we discuss the computation of~\eqref{eq:wantedRatios} and show how the minimization problem~\eqref{eq:choiceit} can be relaxed in order to accelerate the search for suitable indices.

\subsection{Computation of characteristic polynomial coefficients}\label{sec:charpoly}
%
%
%
%

Assuming that the first $t-1$ indices have been selected, we set $B := A - \pi_{i_1, \ldots, i_{t-1}}A$. Then 
\begin{equation*}
    B_i = B - \pi_i B = \left ( I - \frac{b_i b_i^T}{\| b_i \|_2^2} \right ) B
\end{equation*}
is a rank-1 modification of $B$. Deshpande and Rademacher~\cite{Deshpande2010} propose two methods to compute~\eqref{eq:wantedRatios} for $i=1, \ldots, n$. In the following, we summarize them briefly. 
\begin{enumerate}
 \item Algorithm 2 in~\cite{Deshpande2010} computes $BB^T$ explicitly and then computes $B_i B_i^T$ as a rank-2 update of $BB^T$ for every $i=1, \ldots, n$. The characteristic polynomial of $B_i B_i^T$ is computed by establishing a similarity transformation to a matrix in Frobenius normal form~\cite[Section 16.6]{Buergisser1997}. Fast matrix-matrix multiplication and inversion can be exploited so that the cost of this approach is $O(nm^{\omega} \log m)$, where $\omega \le 2.373$ is the best exponent of matrix-matrix multiplication complexity.
 \item Algorithm 3 in~\cite{Deshpande2010} computes the thin SVD of $B = U\Sigma V^T$, the characteristic polynomial of $BB^T$ from the squared singular values of $B$, and the auxiliary polynomials $g_j(x) = \prod_{\ell \neq j} \left (x - \sigma_{\ell}(B)^2\right )$ for $j=1, \ldots, m$. For $h=m-k+t$ and $h=m-k+t-1$, the coefficient $c_h(B_i B_i^T)$ can then be computed as the coefficient of $x^h$ in
 \begin{equation}\label{eq:charpolyupdate}
    \det(\lambda I - BB^T) + \frac{1}{\| b_i \|_2^2} \sum_{j=1}^n \sigma_j^2(B) v_{ij}^2 g_j(x).
 \end{equation}
The cost of this second approach is $O(m^2 n)$.
\end{enumerate}

The problem of computing the Frobenius normal form of a matrix is ``numerically not viable''~\cite{RehmanIpsen2011b}. Also, updating directly the characteristic polynomial as in~\eqref{eq:charpolyupdate} is prone to numerical cancellation, leading to inaccurate results. For instance, consider the $2 \times 2$ matrix
\begin{equation*}
    A = \begin{bmatrix}
            6.583644 \cdot 10^{-7} & 8.113362 \cdot 10^{-3} \\
            8.113362 \cdot 10^{-3} & 100 \\
        \end{bmatrix},
\end{equation*}
and the column selection problem for $k=1$. Algorithm 4 in~\cite{Deshpande2010} using~\eqref{eq:charpolyupdate} selects the first column, giving an error $ \| A - A(:,1)A(:,1)^+ A\|_F \approx 1.2 \cdot 10^{-6} \gg \sqrt{2}\sigma_2(A) = 1.4 \cdot 10^{-10}$.

Therefore, from now on we will avoid updating coefficients of characteristic polynomials and work with singular values instead. More specifically, we will compute the singular values of $B_i$ by updating the SVD of $B$ and then apply the Summation Algorithm~\cite[Algorithm 1]{RehmanIpsen2011} to compute the coefficients of the characteristic polynomial of $B_i B_i^T$ from its eigenvalues (that is, the squared singular values) with $O(m^2)$ operations in a numerically forward stable manner. To describe the updating procedure, consider the (thin) SVD \[
B = U \Sigma V^T,\qquad U \in \R^{m \times m},\quad \Sigma \in \R^{m \times m},\quad V \in \R^{n \times m}.
\]
The (nonzero) singular values of $B_i$ and
\begin{equation*}
    U^T B_i V = (I - U^T \pi_i U ) U^T B V = \left ( I - \frac{U^T b_i b_i^T U}{\| b_i \|_2^2} \right )\Sigma = \left ( I - qq^T \right ) \Sigma,
\end{equation*}
with $q = U^T b_i / \| b_i\|_2$, are identical.
Using standard bulge chasing algorithms (see, e.g.,~\cite[Algorithm 3.4]{Yoon1996} and~\cite{Aurentz2018}) it is possible to find orthogonal matrices $Q,W \in \R^{m \times m}$ such that $Q^T q = e_1$, where $e_1$ denotes the first unit vector, and $Q^T \Sigma W$ is upper bidiagonal. In turn, the singular values can be computed from the bidiagonal matrix 
\begin{equation*}
Q^T (I - qq^T)\Sigma W = (I-e_1 e_1^T)(Q^T\Sigma W).
\end{equation*}
The matrices $Q$ and $W$ are composed of $O(m^2)$ Givens rotations~\cite[Section 5.1]{GolubVanLoan2013} and the computation of $Q^T \Sigma W$ requires to apply each of these rotations to at most $3$ vectors. In turn, the cost of computing this bidiagonal matrix is $O(m^2)$, which is identical with the cost of computing its singular values~\cite[Section 8.6]{GolubVanLoan2013}.

\subsection{Overall algorithm}

The described variation of the column subset selection algorithm by Deshpande and Rademacher is summarized in Algorithm~\ref{alg:CS}. One execution of line~\ref{line:svd} is $O(nm^2)$, lines~\ref{line:beginupdate}--\ref{line:endupdate} are $O(m^2)$, and lines~\ref{line:updBbegin}--\ref{line:updBend} are $O(knm)$. In summary, the overall complexity of Algorithm~\ref{alg:CS} is $O(knm^2)$. This is identical with the complexity of~\cite[Algorithm 4]{Deshpande2010} combined with~\cite[Algorithm 3]{Deshpande2010}, and it is better than~\cite[Algorithm 4]{Deshpande2010} combined with~\cite[Algorithm 2]{Deshpande2010}.

Note that instead of lines~\ref{line:updBbegin}--\ref{line:updBend} we could have updated $B \leftarrow B - \pi_{i_t}B$. However, we noticed that recomputing $B$ in lines 18--19 tends to improve accuracy and it does not change the overall complexity. 

\begin{algorithm} \small 
    \begin{algorithmic}[1]\caption{Column Subset Selection}\label{alg:CS}
        \REQUIRE{$A \in \R^{m\times n}$, rank $1 \le k < m$}
        \ENSURE{Column indices $S \in \{1, \ldots, n\}^k$}
        \STATE{Initialize $S = \emptyset$ and $B = A$}
        \FOR{$t=1, \ldots, k$}
            \STATE{ \label{line:svd} $[U, \Sigma, \sim] = \text{svd}(B)$}
            \STATE{$\minRatio = +\infty$}
            \FOR{$i=1, \ldots, n$}
                \STATE{\label{line:beginupdate} $q = U^T b_i / \| b_i \|_2$}
                \STATE{\label{line:bulge} $D = Q^T \Sigma W $ bidiagonal matrix obtained by bulge chasing~\cite[Algorithm 3.4]{Yoon1996}}
                \STATE{Compute singular values $\sigma_1, \ldots, \sigma_m$ of $(I - e_1 e_1^T)D$}
                \STATE{\label{line:endupdate} Apply Summation Algorithm~\cite[Algorithm 1]{RehmanIpsen2011} to compute $c_{m-k+t-1}(B_i B_i^T)$ and $c_{m-k+t}(B_i B_i^T)$ from eigenvalues $\sigma^2_1, \ldots, \sigma^2_m$}
                \STATE{Set ratio $= c_{m-k+t-1}(B_i B_i^T) / c_{m-k+t}(B_i B_i^T)$}\label{substitutefromhere}
                \STATE \label{line:condition} {\bf if } ratio $ < \minRatio$ {\bf then } Set $\minRatio =$ ratio and $i_t = i$ {\bf end if }
            \ENDFOR
            \STATE{Append index $S \gets (S, i_t)$}
            \STATE{\label{line:updBbegin}$[Q, \sim] = \text{qr}(A(:,S))$}
            \STATE{\label{line:updBend}$B = A - QQ^TA$}
        \ENDFOR
        
    \end{algorithmic}
\end{algorithm}

\subsection{Early stopping of column search}\label{sec:stoppingearly}

For each column index, Algorithm~\ref{alg:CS} needs to traverse $O(n)$ columns in order to find the one that minimizes the coefficient ratio or, equivalently, the conditional expectation. This column search can be shortened. To describe the idea, suppose that $i_1, \ldots, i_{t-1}$ have already been selected such that
\begin{equation} \label{eq:inductionstopearly}
    \mathbb{E} \left [ \| A - \pi_{X_1,\ldots,X_k}A \|_F^2 \mid X_1 = i_1, \ldots, X_{t-1} = i_{t-1}\right ] \le (k+1)(\sigma_{k+1}^2 + \ldots + \sigma_m^2)
\end{equation}
holds.
Now, we can choose \emph{any} $i_t$ such that
\begin{equation}\label{eq:relaxed}
    \mathbb{E} \left [ \| A - \pi_{X_1,\ldots,X_k}A \|_F^2 \mid X_1 = i_1, \ldots, X_{t} = i_{t}\right ] \le (k+1)(\sigma_{k+1}^2 + \ldots + \sigma_m^2)
\end{equation}
holds. The existence of $i_t$ is guaranteed by~\eqref{eq:inductionstopearly} but we do not need to find the one that minimizes the conditional expectation. It suffices to always choose in every step an index such that~\eqref{eq:relaxed} is verified. By induction, the error bound~\eqref{eq:CS} still holds.

The discussion above suggests to modify
Algorithm~\ref{alg:CS} such that it computes
    \begin{algorithmic}
        \STATE{$\upperBound = (k+1)\cdot \left ( \sigma_{k+1}^2 + \ldots + \sigma_m^2 \right )$}
    \end{algorithmic}
in the beginning and substitute line~\ref{line:condition} with
    \begin{algorithmic}[1]
    \setcounterref{ALC@line}{substitutefromhere}
     \STATE {\bf if } $(k-t+1) \cdot \text{ratio} \le \upperBound$ {\bf then } Set $i_t = i$ and break {\bf end if }
         \end{algorithmic}
To be able to stop the search early, it is important to test the columns in a suitable order. We found it beneficial to test the columns of $B$ in descending norm. For each step $t$, computing the norms of all columns of $B$ and sorting them has complexity $O(mn + n \log n)$. Although this choice is clearly heuristic, the following lemma provides some justification for it by showing that the column of largest norm is the right choice for $k = 1$ provided that all other columns are sufficiently small. 
\begin{lemma} \label{lemma:indication}
    Let $A = \begin{bmatrix} a_1 & A_2 \end{bmatrix}$. If $ \| A_2 \|_F \le \| a_1 \|_2$ then choosing the first column solves the column selection problem for $k=1$, that is,
    \begin{equation*}
        \| A - a_1 a_1^+ A \|_F^2 \le 2(\sigma_2^2 + \ldots + \sigma_m^2).
    \end{equation*}
\end{lemma}
Note that the condition of the lemma is satisfied if the column norms of $A$ decay sufficiently fast, for instance if $\| a_i \|_2 \le \frac{\| a_1 \|_2}{i}$ for $i=2, \ldots, n$.
\begin{proof}
Without loss of generality we may assume that $\|a_1\|_2 = 1$. By setting $B = A_2 - a_1 a_1^+ A_2 = A_2 - a_1 a_1^T A_2$ and $b = A_2^T a_1$, we have
\[
 A^T A = \begin{bmatrix} 1 & b^T \\ b & A_2^T A_2 \end{bmatrix} = \begin{bmatrix} 1 & b^T \\ b & B^T B + bb^T \end{bmatrix}
\]
and obtain \begin{equation} \label{eq:boundlala}
 \|A^T A\|_2 \le \left\| \begin{bmatrix} 1 & \|b\|_2 \\ \|b\|_2 & \| B^T B + bb^T \|_2 \end{bmatrix} \right\|_2
 \le \left\| \begin{bmatrix} 1 & \|b\|_2 \\ \|b\|_2 & \| B\|_F^2 + \|b \|^2_2 \end{bmatrix} \right\|_2.
\end{equation}
Here, the first inequality is a norm-compression inequality~\cite[Section 9.10]{Bernstein2009} and the second inequality follows from the fact that the involved matrices are positive.

We aim at proving
    \begin{equation*}
        \| A - a_1 a_1^+ A \|_F^2 = \| B \|_F^2 \le 2(\| A \|_F^2 - \| A \|_2^2),
    \end{equation*}
    which is equivalent to
    \begin{equation*}
    \| A \|_2^2 \le 1 + \| b \|_2^2 + \frac{\| B \|_F^2}{2}=:\gamma.
    \end{equation*}
Thus, it remains to show that the larger eigenvalue of the symmetric positive definite $2\times 2$ matrix on the right-hand side of~\eqref{eq:boundlala} is bounded by $\gamma$. For this purpose, we note that its characteristic polynomial is given by
\[
p(\lambda) = \left ( \lambda - 1\right ) \left ( \lambda - \|b\|_2^2 - \|B\|_F^2 \right ) - \|b\|_2^2.
\]
Setting $\gamma = 1 + \| b \|_2^2 + \| B \|_F^2/2$, we obtain
\[
 p(\gamma) = \| B \|_F^2/2 \cdot \left ( 1 - \|b\|_2^2 - \|B\|_F^2/2 \right ) \ge 0,
\]
where we used that $\| b \|_2^2 + \| B \|_F^2 = \| A_2 \|^2_F \le \| a_1 \|^2_2 = 1$. Because $p$ is a parabola with vertex $(1+\|b\|_2^2 + \|B\|_F^2)/2\le \gamma$, it follows that the larger root of $p$ is bounded by $\gamma$, which completes the proof. 
\end{proof}

It is important to not draw too many conclusions from Lemma~\ref{lemma:indication}. Consider, for example, the matrix
\begin{equation*}
    A = \begin{bmatrix} 1 & 0 & 10^{-b} \\ 0 & 1 & 10^{-b} \\ 0 & 0 & 10^{-2b} \end{bmatrix}
\end{equation*}
for some integer $b$, say $b = 16$. For $k = 1$, the optimal choice is the third column, which is the one of smallest norm. This matrix also nicely illustrates that the obvious greedy approach (in order to get $k$ columns of $A$, one first chooses the best column, then the best column in the orthogonal complement, and so on) comes with no guarantees and may, in fact, utterly fail.
For $k = 2$ the optimal choice consists of the first two columns. On the other hand, the greedy approach for $k=2$ first selects the third column and then the first column, resulting in the arbitrarily bad error ratio $\frac{\text{error greedy}}{\text{error best}} \approx 10^b$. 

\subsection{Numerical experiments}\label{sec:numexpCS}

Both variants of Algorithm~\ref{alg:CS}, without and with early stopping, have been implemented in Matlab version R2019a.
As the bulge chasing algorithm in line~\ref{line:bulge} would perform poorly in Matlab, this part has been implemented in C++ and is called via a MEX interface.
All numerical experiments in this work have been run on an eight-core Intel Core i7-8650U 1.90 GHz CPU, 256 KB of level 2 Cache and 16 GB of RAM. Multithreading has been turned off in order to not distort the findings. 

We have applied the algorithm to the following three matrices:
\begin{enumerate}
    \item the Hilbert matrix $A_1 \in \R^{200 \times 200}$ given by $A_1(i,j) = \frac{1}{i+j-1}$;
    \item $A_2 \in \R^{100 \times 200}$ given by $A_2(i,j) = \exp(-0.3 \cdot |i-j|/200)$;
    \item $A_3 \in \R^{100 \times 200}$ given by $A_3(i,j) = \left ( \left ( \frac{i}{200}\right )^{20} + \left ( \frac{j}{200}\right )^{20} \right )^{1/20}$.
\end{enumerate}
The obtained results are shown in Figures~\ref{fig:hilbert200CS},~\ref{fig:exp100200CS}, and~\ref{fig:p100200CS} respectively. Each left plot contains, for different values of $k$, the approximation error $\| A_i - A_i(:,S)A_i(:,S)^+ \|_F$ returned by Algorithm~\ref{alg:CS}, without and with early stopping. We compare with the best rank-$k$ approximation error $\sqrt{\sigma_{k+1}^2 + \ldots + \sigma_m^2}$ and the upper bound~\eqref{eq:CS}, that is, $\sqrt{(k+1)(\sigma_{k+1}^2 + \ldots + \sigma_m^2)}$. It can be seen that both variants of our algorithm stay below the upper bound, until it reaches the level of roundoff error. Interestingly, for the matrix $A_2$, which features the slowest singular value decay, the observed approximation error is much closer to the best approximation error than to the upper bound.
The right plots of the figures show, for different values of $k$, the ratio between the execution times of Algorithm~\ref{alg:CS} without early stopping and with early stopping. For the variant with early stopping, we also plot the number of columns that were examined. In the most optimistic scenario, only $k$ columns need to be examined, which means that in every step of the algorithm already the first verifies the desired criterion. The plots reveal that our algorithm actually stays pretty close to this ideal situation, at least for the matrices considered. Note that for values of $k$ larger than the numerical rank of the matrix, Algorithm~\ref{alg:CS} starts computing ratios~\eqref{eq:coeffCharPoly} from singular values of the order of machine precision. In turn, the computations are severely affected by roundoff error and it may, in fact, happen that the early stopping criterion is never satisfied. This leads to meaningless results and we therefore truncate the plots before this happens. A proper implementation of Algorithm~\ref{alg:CS} needs to detect such a situation and reduce $k$ accordingly.

\begin{figure}[ht]
    \centerline{\includegraphics[width=\textwidth]{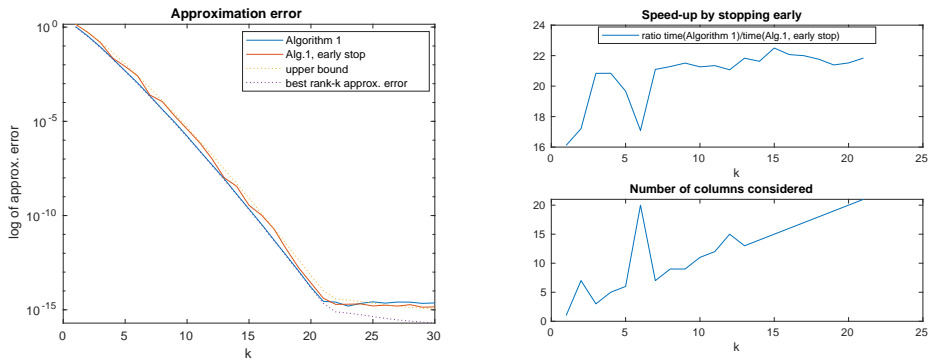}}
    \caption{Results for matrix $A_1$}
    \label{fig:hilbert200CS}
\end{figure}

\begin{figure}[ht]
    \centerline{\includegraphics[width=\textwidth]{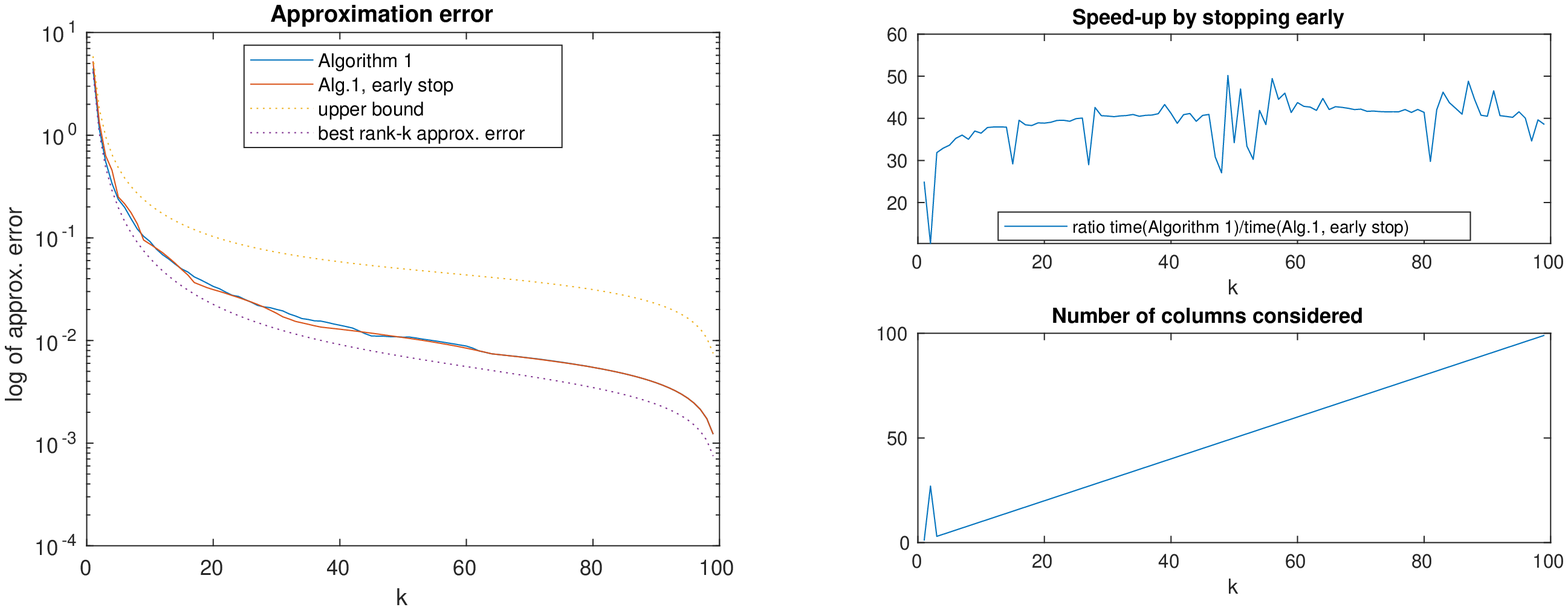}}
    \caption{Results for matrix $A_2$}
    \label{fig:exp100200CS}
\end{figure}

\begin{figure}[ht]
    \centerline{\includegraphics[width=\textwidth]{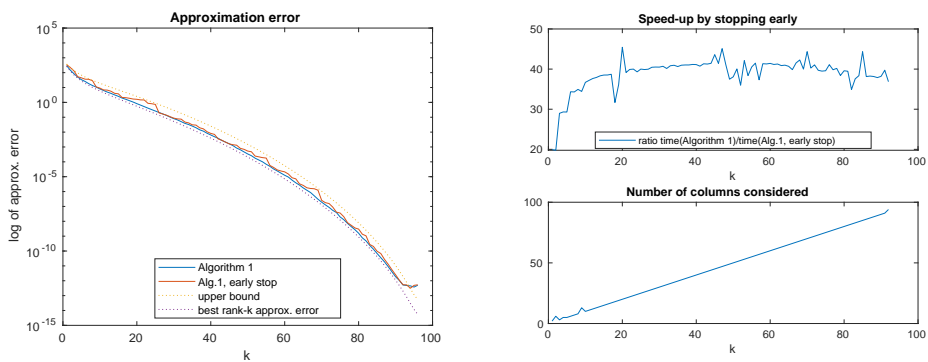}}
    \caption{Results for matrix $A_3$}
    \label{fig:p100200CS}
\end{figure}

\section{Matrix approximation} \label{sec:matrixapprox}

In this section, we extend the developments from Section~\ref{sec:colsel} on column subset selection to compute certain low-rank matrix approximations of a matrix $A\in \R^{m\times n}$ with $m\le n$. As already discussed in the introduction, we will pursue two ways. First, in Section~\ref{sec:cur}, we discuss a general CUR approximation obtained from applying column subset selection to the columns and rows of the matrix. Second, in Section~\ref{sec:cross}, we present a novel approach to cross approximation, a specific type of CUR approximation, with guaranteed error bounds.

%

\subsection{CUR approximation induced by column subset selection} \label{sec:cur}

Suppose that $C \in \R^{m\times k}$ and $R \in \R^{k\times n}$ have been chosen. Then the matrix $U\in\R^{k \times k}$ that minimizes $\| A - CUR \|_F$ is given by the projection $U = C^+ A R^+$, see~\cite[p. 320]{Stewart1999}. The following corollary provides an error bound for the case when $C$ and $R$ are determined by the techniques from Section~\ref{sec:colsel}, leading to  Algorithm~\ref{alg:MCS}. Closely related results can be found in the literature; see, for example,~\cite[Theorem 4]{Drineas2008},~\cite[Corollary 3.5]{Saibaba2016}, and~\cite[Theorem 4.1]{SorensenEmbree2016}. 
\begin{algorithm} \small 
    \begin{algorithmic}[1]\caption{Matrix approximation by column subset selection}\label{alg:MCS}
        \REQUIRE{$A \in \R^{m \times n}$, rank $k$}
        \ENSURE{Rank-$k$ CUR approximation, with $C,R$ containing columns and rows of $A$}
        \STATE{Compute $C$ by applying Algorithm \ref{alg:CS} to select $k$ columns of $A$}
        \STATE{Compute $R$ by applying Algorithm \ref{alg:CS} to select $k$ columns of $A^T$}
        \STATE{Compute $U = C^{+}AR^{+}$}
    \end{algorithmic}
\end{algorithm}
\begin{corollary}\label{cor:CUR}
    Let $A \in \R^{m \times n}$, with $1 \le k \le m \le n$. Then the CUR approximation returned by Algorithm~\ref{alg:MCS} satisfies
    \begin{equation*}
        \| A - CUR \|_F \le \sqrt{2k+2} \sqrt{ \sigma_{k+1}(A)^2 + \cdots  + \sigma_{m}(A)^2  }.
    \end{equation*}
\end{corollary}
\begin{proof}
    Using the inequality~\eqref{eq:CS} twice and the fact that $CC^{+}$ is an orthogonal projection, we obtain
    \begin{align*}
        \| A - CUR \|_F^2 & = \| A - CC^{+}AR^{+}R \|_F^2 \\
        & = \| A - CC^{+} A \|_F^2 + \| CC^{+}(A - AR^{+}R) \|_F^2 \\
        & \le \| (I - CC^{+})A \|_F^2 + \| A(I - R^{+}R) \|_F^2 \\
        & \le 2(k+1) \left ( \sigma_{k+1}(A)^2 + \cdots  + \sigma_{m}(A)^2 \right ).
    \end{align*}
\end{proof}

\subsubsection{Numerical experiments}

We have tested a Matlab implementation of Algorithm~\ref{alg:MCS} in the setting and for the matrices $A_1,A_2,A_3$ described in Section~\ref{sec:numexpCS}. Figure~\ref{fig:plotsCSM} displays the obtained approximation errors $\| A_i - CUR \|_F$ for different values of $k$. Again, we have tested both variants of Algorithm~\ref{alg:CS}, without and with early stopping, within Algorithm~\ref{alg:MCS}. The speedups obtained from early stopping are very similar to the ones reported Section~\ref{sec:numexpCS} and, therefore, we refrain from providing details.

\begin{figure}[ht]
    \centerline{\includegraphics[width=\textwidth]{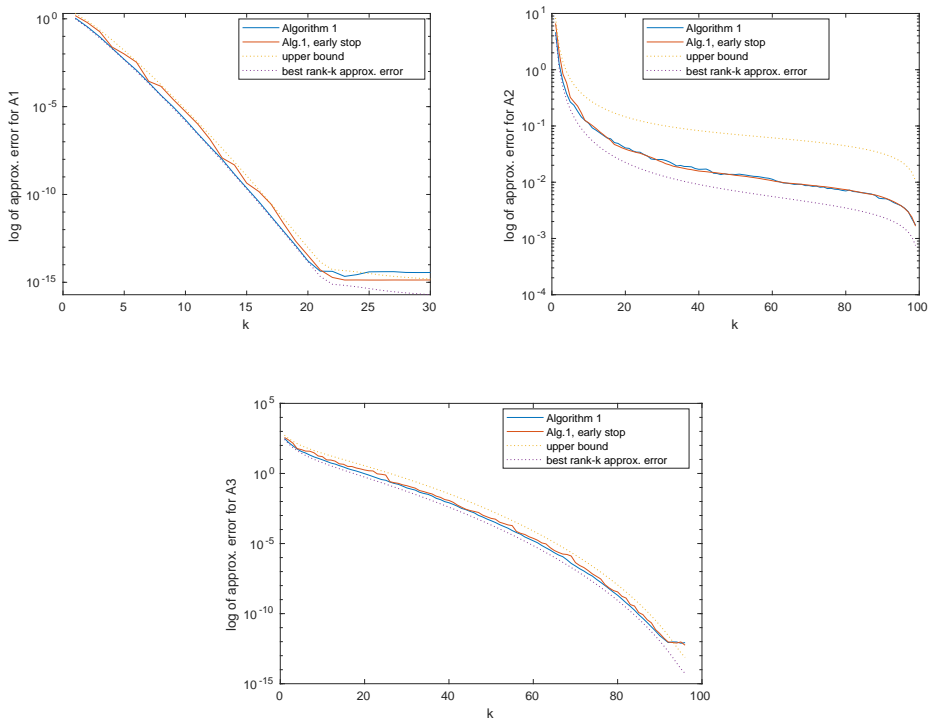}}
    \caption{Approximation errors for matrices $A_1$ (top left), $A_2$ (top right), and $A_3$ (bottom).}
    \label{fig:plotsCSM}
\end{figure}

We also consider, for $0<\alpha < 1$, the $n \times n$ matrix  
\begin{equation*}
A = Q\cdot \text{diag}(1, \alpha, \alpha^2, \ldots, \alpha^{n-1}) \cdot Q^T,
\end{equation*}
where $Q \in \R^{n \times n}$ is determined as the orthogonal factor from the QR decomposition of \[
\begin{bmatrix} 1 & & & & \\ -1 & 1 & & & \\ -1 & -1 & 1 & & \\ \vdots & \vdots & & \ddots & \\ -1 & -1 & -1 & \cdots & 1 \end{bmatrix}.
\]
This is known to be a challenging example for the CUR approximation induced by DEIM (discrete interpolation method); see~\cite[Section 4.2]{SorensenEmbree2016}, which determines the row and column indices by greedily choosing a maximum volume submatrix of $U_k$ and $V_k$ containing the first $k$ left and right singular vectors of $A$, respectively. For the example above, the DEIM induced CUR approximation always chooses $1,\ldots,k$ for the column and row indices. For
$\alpha=0.1$, $n=6$, $k=5$, the error resulting from this choice is given by
\begin{equation*}
    \| A - A(:,1:5)A(:,1:5)^+AA(1:5,:)^+ A(1:5,:) \|_F \approx 2.6 \cdot 10^{-9},
\end{equation*}
which is a magnitude larger than the upper bound $\sqrt{2(k+1)} \sigma_6\approx 3.5 \cdot 10^{-10}$ guaranteed by Algorithm~\ref{alg:MCS}. Note that the latter algorithm selects the last $5$ rows and columns for this example, leading to an error of $\approx 1.3 \cdot 10^{-10}$.

\subsection{Cross approximation} \label{sec:cross}

We now consider \emph{cross approximations}, which take the form
\begin{equation} \label{eq:crossapproximation}
    A\approx A(:,J) A(I,J)^{-1} A(I,:)
\end{equation}
for row/column index tuples \[ (I,J) \in \Omega:= \{1, \ldots, m\}^k \times \{ 1, \ldots, n\}^k. \]
The different choice of the middle matrix makes a fundamental difference. In particular, as the following example shows, choosing the indices $I,J$ as in Algorithm~\ref{alg:MCS} may lead to poor approximation error.
\begin{example}
Consider $A = \begin{bmatrix} 2\varepsilon & 1 \\ 1 & \varepsilon \end{bmatrix}$ for $\varepsilon > 0$ and $k = 1$. Clearly, the first column and row satisfy the bound~\eqref{eq:CS} for $k=1$ with respect to $A$ and $A^T$, respectively. However, the error of the corresponding cross approximation, $\|A - A(:,1)A(1,1)^{-1}A(1,:)\|_F =\frac{1}{2\varepsilon} - \varepsilon $, becomes arbitrarily large as $\varepsilon \to 0$. 
\end{example}

Zamarashkin and Osinsky~\cite{Zamarashkin2018} have shown the existence of a cross approximation that satisfies a polynomial error bound in the Frobenius norm. To summarize their result, let 
\begin{equation*}
(X, Y) = (X_1, \ldots, X_k,Y_1, \ldots, Y_k)
\end{equation*}
be a $(2k)$-tuple of random variables with values in $\Omega$ such that
\begin{equation}\label{eq:VolSamplingACA}
    \mathbb{P}\left ( X = I, Y = J \right) := 
                                    \frac{\Vol\left(A( I,J)\right)^2}{\sum_{(I', J') \in \Omega} \Vol(A(I',J'))^2}.
\end{equation}
Note that $\Vol\left(A(I,J)\right) = 0$ whenever $i_1,\ldots, i_k$ or $j_1, \ldots, j_k$ contain repeated indices. Then~\cite[Theorem 1]{Zamarashkin2018} shows that
\begin{equation}\label{eq:ExpectedACA}
    \mathbb{E}[ \| A - A(:,Y)A(X,Y)^{-1}A(X,:) \|_F^2 ] \le (k+1)^2 \left (\sigma_{k+1}^2 + \ldots + \sigma_m^2\right ).
\end{equation}
In particular, this implies that there exists $(I, J) \in \Omega$ such that
\begin{equation}\label{eq:ACAfrob}
    \| A - A(:,J)A(I,J)^{-1}A(I,:) \|_F^2 \le (k+1)^2\left(\sigma_{k+1}^2 + \ldots + \sigma_m^2 \right ).
\end{equation}
 
In analogy to Section~\ref{sec:detAlgorithm} and~\cite{Deshpande2010}, we will now derandomize this result producing a polynomial-time deterministic algorithm that returns a cross approximation satisfying~\eqref{eq:ACAfrob}. The key for doing so is to find an expression for the conditional expectations that is easy to work with.

\subsubsection{Conditional expectations}
 
 \begin{lemma}\label{lem:marginalsACA}
    Let $1\le t \le k$ and $(i_1, \ldots, i_t, j_1, \ldots, j_t)$ be such that 
    \begin{equation*}
    \mathbb{P}\left ( \scriptstyle{X_1=i_1, \ldots, X_t = i_t \atop Y_1 = j_1, \ldots, Y_t = j_t} \right)  > 0
    \end{equation*}
    for a random $(2k)$-tuple $(X, Y)$ with the probability distribution defined by~\eqref{eq:VolSamplingACA}. Consider
    \begin{equation*}
        B = A - A(:,\small{\begin{bmatrix} j_1 & \cdots & j_t \end{bmatrix}}) A(\begin{bmatrix} i_1 & \cdots & i_t \end{bmatrix}, \begin{bmatrix} j_1 & \cdots & j_t \end{bmatrix})^{-1} A(\begin{bmatrix} i_1 & \cdots & i_t \end{bmatrix},:),
    \end{equation*}
    the remainder of cross approximation after choosing row indices $i_1, \ldots, i_t$ and column indices $j_1, \ldots, j_t$. Then
    \begin{align*}
        &\mathbb{E}\left [  \| A - A(:,Y)A(X,Y)^{-1}A(X,:)\|_F^2 \big| \scriptstyle{{X_1=i_1,\ldots,X_t=i_t \atop Y_1=j_1,\ldots,Y_t=j_t } }\right ]\\
        & = (k-t+1)^2 \cdot \frac{c_{m-k+t-1}(BB^T)}{c_{m-k+t}(BB^T)},
    \end{align*}
    with the coefficients $c_{m-k+t},c_{m-k+t-1}$ defined as in~\eqref{eq:coeffCharPoly}.
 \end{lemma}
 
 \begin{proof}
    To simplify notation, we let $I_1 = (i_1, \ldots , i_t)$, $I_2 = (i_{t+1}, \ldots, i_k)$, $I = (I_1, I_2) = (i_1,\ldots,i_k)$ and define $J_1, J_2, J$ analogously. In the following, we always use the convention that row and column summation indices range from $1$ to $m$ and from $1$ to $n$, respectively.
    We have that
    \begin{align}
        &\mathbb{E}\left [  \| A - A(:,Y) A(X,Y)^{-1} A(X,:) \|_F^2 \big| { \scriptstyle { X_1=i_1,\ldots,X_t=i_t \atop Y_1=j_1,\ldots,Y_t=j_t}} \right ] \nonumber \\
        & = \sum_{i_{t+1}, \ldots, i_k \atop j_{t+1}, \ldots, j_k} \| A - A(:,J) A(I,J)^{-1} A(I,:) \|_F^2 \cdot \mathbb{P}\left(X = I, Y=J \big| { \scriptstyle { X_1=i_1,\ldots,X_t=i_t \atop Y_1=j_1,\ldots,Y_t=j_t}}\right) \nonumber \\
        & = \frac{1}{\gamma} \cdot \sum_{i_{t+1}, \ldots, i_k, i_{k+1} \atop j_{t+1}, \ldots, j_k, j_{k+1}} 
         \Vol \big(A((I, i_{k+1}), (J,j_{k+1})) \big)^2, \label{eq:lastline}
    \end{align}
    with 
    \begin{equation*}
        \gamma = \sum_{i_{t+1}, \ldots, i_k \atop j_{t+1}, \ldots, j_k} \Vol(A(I,J))^2.
    \end{equation*}
    For establishing the equality in~\eqref{eq:lastline} we used from~\cite[Lemma 1]{Zamarashkin2018} that
    \begin{equation*}
        \| A - A(:,J) A(I,J)^{-1} A(I,:) \|_F^2 = \frac{\sum_{i_{k+1}, j_{k+1}} \!\!\!{\Vol\big(A((I, i_{k+1}), (J,j_{k+1})) \big)^2}} {\Vol(A(\I,\J))^2},
    \end{equation*}
    and, from~\eqref{eq:VolSamplingACA}, that
    \begin{equation*}
        \mathbb{P}\left (X = I, Y=J \big | { \scriptstyle { X_1=i_1,\ldots,X_t=i_t \atop Y_1=j_1,\ldots,Y_t=j_t}}\right ) = \frac{\mathbb{P}(X=I, Y=J)}{\mathbb{P}\left(\scriptstyle{X_1=i_1,\ldots,X_t=i_t \atop Y_1=j_1,\ldots,Y_t=j_t}\right )} =  \frac{1}{\gamma} \cdot \Vol(A(\I,\J))^2.
    \end{equation*}
    
    We now aim at simplifying the expression~\eqref{eq:lastline}. For this purpose, we assume without loss of generality that $i_1 = 1,\ldots,i_t = t$ and $j_1 = 1,\ldots,j_t = t$. This allows us to partition
    \[
     A(I,J) = \begin{bmatrix}
                 A(I_1,J_1) & A(I_1, J_2) \\
                 A(I_2, J_1) & A(I_2,J_2)
                \end{bmatrix}, \quad 
B(I,J) = \begin{bmatrix}
                 0 & 0 \\
                 0 & B(I_2,J_2)
                \end{bmatrix},
    \]
    where $B(I_2,J_2) = A(I_2,J_2) - A(I_2, J_1) A(I_1,J_1)^{-1} A(I_1, J_2)$ by the definition of $B$. By the relation between determinants and Schur complements~\cite[Equation (0.8.5.1)]{HornJohnson2013},
    $\Vol( A(I,J) ) = \Vol( A(I_1,J_1) ) \cdot \Vol( B(I_2,J_2) )$. Therefore,
    \begin{equation*}
        \gamma = \sum_{i_{t+1}, \ldots, i_k \atop j_{t+1}, \ldots, j_k} \Vol(A(\I,\J))^2 = \sum_{i_{t+1}, \ldots, i_k \atop j_{t+1}, \ldots, j_k} \Vol(B(\I_2,\J_2))^2 \cdot \Vol(A(I_1, J_1))^2.
    \end{equation*}
    Analogously, one shows
    \begin{multline*}
        \sum_{i_{t+1}, \ldots, i_{k+1} \atop j_{t+1}, \ldots, j_{k+1}} \Vol\big(A((I,i_{k+1}),(J,j_{k+1}))\big)^2  \\ = 
        \sum_{i_{t+1}, \ldots, i_{k+1} \atop j_{t+1}, \ldots, j_{k+1}}
        \Vol\big(B((I_2, i_{k+1}), (J_2, j_{k+1}))\big)^2 \Vol(A(I_1, J_1))^2.
    \end{multline*}   
    Inserting these expressions into~\eqref{eq:lastline} yields
    \begin{equation*}
       \frac{ \sum_{i_{t+1}, \ldots, i_{k+1} \atop j_{t+1}, \ldots, j_{k+1}} \Vol\big(B((I_2, i_{k+1}), (J_2, j_{k+1}))\big)^2}{ \sum_{i_{t+1}, \ldots, i_k \atop j_{t+1}, \ldots, j_k} \Vol(B(\I_2,\J_2))^2}.
    \end{equation*}
    By~\cite[Theorem 7]{Knill2014} this ratio is equal to 
    \begin{equation*}
        \frac{ c_{m-k+t-1}(BB^T) \cdot \left ( (k-t+1)! \right )^2}{c_{m-k+t}(BB^T) \cdot \left ( (k-t)! \right )^2} =  (k-t+1)^2 \cdot \frac{c_{m-k+t-1}(BB^T)}{c_{m-k+t}(BB^T)}.\qedhere
    \end{equation*}
 \end{proof}

 \subsubsection{Derandomized cross approximation algorithm}
With Lemma~\ref{lem:marginalsACA} at hand, we can proceed analogously to Section~\ref{sec:detAlgorithm} and sequentially find $k$ pairs of row/column indices such that~\eqref{eq:ACAfrob} is satisfied. Suppose that $t-1$ index pairs $(i_1,j_1),\ldots,(i_{t-1},j_{t-1})$ have been determined. Then the $t$th step of the algorithm proceeds by choosing $(i_t, j_t)$ such that 
\begin{equation}\label{eq:choiceitjt}
    \mathbb{E}\left [ \| A-A(:,Y)A(X,Y)^{-1} A(X,:) \|_F^2 \big | \scriptstyle{X_1=i_1, \ldots, X_t=i_t \atop Y_1 = j_1, \ldots, Y_t = j_t}\right ]
\end{equation}
is minimized. We will show in Theorem~\ref{thm:ACAworks} below that this choice of index pairs leads to a cross approximation satisfying the desired error bound~\eqref{eq:ACAfrob}.
In view of Lemma~\ref{lem:marginalsACA}, the minimization of~\eqref{eq:choiceitjt} means that in each step of the algorithm we need to compute the ratios
 \begin{equation}\label{eq:wantedRatiosACA}
    \frac{c_{m-k+t-1}(C_{ij} C_{ij}^T)}{c_{m-k+t}(C_{ij} C_{ij}^T)}, \quad i = 1,\ldots, m, \quad j = 1,\ldots,n,
 \end{equation}
where 
\begin{equation*}
C_{ij} = A - A(:,\small{[j_1,\cdots,j_{t-1}, j]}) A([i_1,\cdots,i_{t-1},i ], [j_1,\cdots,j_{t-1}, j])^{-1} A([ i_1, \cdots,i_{t-1}, i],:).
\end{equation*}
 
Parallelizing the developments in Section~\ref{sec:charpoly}, we now show how the coefficients in~\eqref{eq:wantedRatiosACA} can be computed via updating the singular values of $C_{ij}$. Let us denote the remainder from the previous step by
 \begin{equation*}
        B = A - A(:,\small{[j_1,\cdots,j_{t-1}]}) A([i_1,\cdots,i_{t-1} ], [j_1,\cdots,j_{t-1}])^{-1} A([ i_1, \cdots,i_{t-1}],:).
 \end{equation*}
Then it follows that \begin{equation} \label{eq:cij} C_{ij} = B - \frac{1}{B(i,j)} B(:,j) B(i,:),\end{equation} see, e.g.,~\cite{Bebendorf2000}. We compute a thin SVD $B = U \Sigma V^T$ such that $U\in \R^{m\times m}$, $V\in \R^{m \times n}$ have orthonormal columns and $\Sigma \in \R^{m\times m}$ is diagonal.
 Note that
 \begin{equation*}
    B(:,j) = U\Sigma V(j,:)^T, \qquad B(i,:) = U(i,:) \Sigma V^T.
 \end{equation*}
 Inserted into~\eqref{eq:cij}, this shows that the nonzero singular values of $C_{ij}$ match the singular values of
 \begin{equation*}
    U^T C_{ij} V = \Sigma - \Sigma V(j,:)^T \cdot \frac{U(i,:) \Sigma}{B(i,j)} = \Sigma - x y^T,
 \end{equation*}
 where $x = \Sigma V(j,:)^T$ and $y = \frac{1}{B(i,j)}\Sigma U(i,:)^T$ are vectors of length $m$ and can be computed with $O(m^2)$ operations. 
 
 Similarly as in Section~\ref{sec:charpoly}, we transform $\Sigma - xy^T$ into bidiagonal form, after which its singular values can be computed with $O(m^2)$ operations. This  transformation proceeds in three steps:
 \begin{enumerate}
 \item We compute orthogonal matrices $Q$ and $W$ such that $Q^T \Sigma W$ is upper bidiagonal and $Q^T x = \pm \| x \|_2 \cdot e_1$ using, for example,~\cite[Algorithm 3.4]{Yoon1996}. In turn, the matrix 
 \begin{equation}\label{eq:bidiagfirstrow}
 D_1 := Q^T (\Sigma - xy^T)W
 \end{equation}
 is bidiagonal with an additional nonzero first row; see the first plot in Figure~\ref{fig:chasing} for an illustration.
 \item By a bulge chasing algorithm, we transform $D_1$ to an upper banded matrix $D_2$ with two superdiagonals using $O(m^2)$ Givens rotations. We refrain from giving a detailed description of the algorithm and refer to Figure~\ref{fig:chasing} for an illustration. 
 \item The banded matrix $D_2$ is reduced to a bidiagonal matrix $D_3$ using the LAPACK~\cite{Anderson1999} routine {\tt dgbbrd}.
  \end{enumerate}
 The overall procedure described above can be implemended by means of $O(m^2)$ Givens rotations, each of which is applied to a small matrix of size independent of $m,n$. Hence, it has complexity $O(m^2)$.
 
 \begin{figure}[ht]
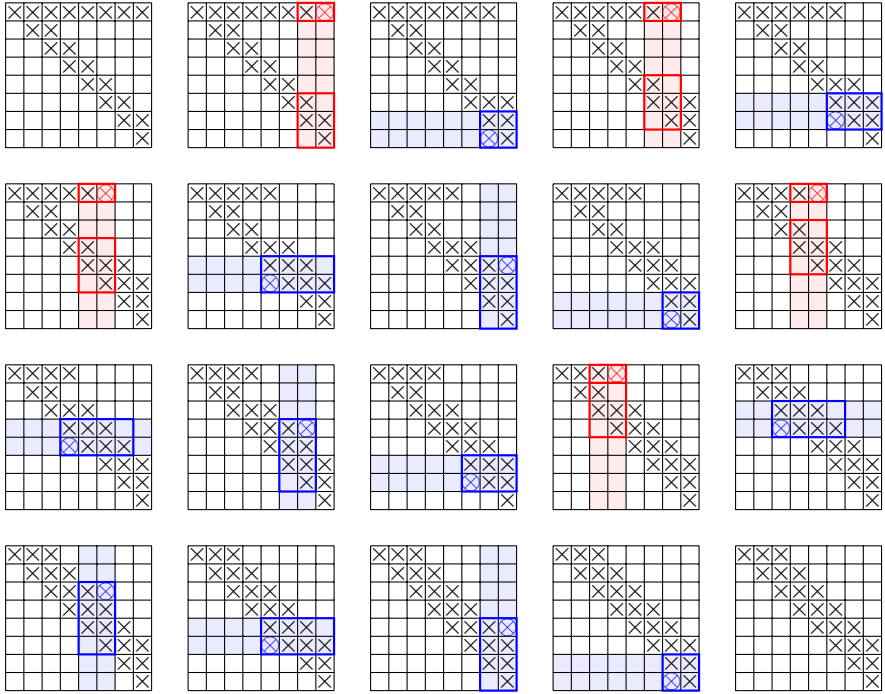

    \begin{center}
        \caption{Illustration of bulge chasing algorithm to transform a bidiagonal matrix with an additional nonzero first row to an upper banded matrix. In each plot, except for the first and last ones, a Givens rotation is applied to a pair of row or columns to zero out the entry denoted by $\otimes$.}
        \label{fig:chasing}
        \include{chasing}
    \end{center}
 \end{figure}

 Algorithm~\ref{alg:derandomizedACA} summarizes our newly proposed method for cross approximation. The SVD needed at the beginning of each outer loop is of complexity $O(m^2n)$ and each of the $mn$ inner loops costs $O(m^2)$ operation; the total complexity of Algorithm~\ref{alg:derandomizedACA} is therefore $O(km^3n)$. 
 \begin{algorithm} \small
    \begin{algorithmic}[1]\caption{Derandomized cross approximation}
        \label{alg:derandomizedACA}
        \REQUIRE{$A\in\R^{m \times n}$ with $m\le n$, integer $k \le m$}
        \ENSURE{Index sets $I,J$ of cardinality $k$ defining the cross approximation~\eqref{eq:crossapproximation}}
        \STATE{Initialize $I \leftarrow \emptyset, J \leftarrow \emptyset$, and $ B \leftarrow A$}
        \FOR{$t = 1,\ldots, k$}
            \STATE{$[U, \Sigma, V] =$ thin SVD of $B$} 
            \STATE{$\minRatio = +\infty$}
            \FOR{$i = 1, \ldots, m$}
                \FOR{$j = 1, \ldots, n$}
                    \STATE{$x \leftarrow \Sigma V(j,:)^T $, $y \leftarrow \frac{1}{B(i,j)} \Sigma U(i,:)^T$}
                    \STATE{\label{line:mex1} Compute matrix $D_1$ defined in~\eqref{eq:bidiagfirstrow} using~\cite[Algorithm 3.4]{Yoon1996} }
                    \STATE{\label{line:mex2} Transform $D_1$ into upper banded form $D_2$ using bulge chasing algorithm}
                    \STATE{Transform $D_2$ into bidiagonal matrix $D_3$ using LAPACK's {\tt dgbbrd}}
                    \STATE{Compute singular values $\sigma_1, \ldots, \sigma_m$ of $D_3$}
                    \STATE{Apply Summation Algorithm~\cite[Algorithm 1]{RehmanIpsen2011}} to obtain $c_{m-k+t-1}(C_{ij} C_{ij}^T)$ and $c_{m-k+t}(C_{ij} C_{ij}^T)$ from eigenvalues $\sigma^2_1, \ldots, \sigma^2_m$
                    \STATE{Set $r = \frac{c_{m-k+t-1}(C_{ij} C_{ij}^T)}{c_{m-k+t}(C_{ij} C_{ij}^T)}$}
                    \STATE {\bf if } $r <  \minRatio$ {\bf then } $\newRow \leftarrow i$, $\newCol \leftarrow j$, $\minRatio = r$ {\bf end if}
                \ENDFOR
            \ENDFOR
            \STATE{$I \leftarrow I \cup \{\newRow\}, J \leftarrow J \cup \{\newCol\}$}
            \STATE{$B \leftarrow B - \frac{B(:,\newCol)\cdot B(\newRow,:)}{B(\newRow,\newCol)}$}
        \ENDFOR
    \end{algorithmic}
 \end{algorithm}
 
  \begin{theorem}\label{thm:ACAworks}
 For a matrix $A$ of rank at least $k$, Algorithm~\ref{alg:derandomizedACA} returns index sets $I$ and $J$ such that~\eqref{eq:ACAfrob} is satisfied.
 \end{theorem}
 \begin{proof}
 Let $B_{\{X,Y\}} = A-A(:,Y)A(X,Y)^{-1} A(X,:)$. For $t=1, \ldots, k$ we have that
\begin{align*}
        & \mathbb{E}\left [ \| B_{\{X,Y\}} \|_F^2 \big | \scriptstyle{   X_1=i_1, \ldots,X_{t-1} = i_{t-1} \atop Y_1 = j_1, \ldots, Y_{t-1} = j_{t-1}} \right ] \\
        &\!\!= \!\sum_{i, j} \mathbb{E} \left [ \| B_{\{X,Y\}} \|_F^2 \big | \scriptstyle{X_1 = i_1, \ldots, X_{t-1} = i_{t-1}, X_t = i \atop Y_1 = j_1, \ldots, Y_{t-1} = j_{t-1}, Y_{t} =j} \right ] \!\mathbb{P}\!\left (\!X_t = i, Y_t = j_t \big | \scriptstyle{ X_1 = i_1, \ldots, X_{t-1} = i_{t-1} \atop Y_1 = j_1, \ldots, Y_{t-1} = j_{t-1}}\right )\!.
    \end{align*}
 Therefore, as~\eqref{eq:ExpectedACA} holds, the choice~\eqref{eq:choiceitjt} inductively ensures that 
 \begin{align*}
    \mathbb{E}\left [ \| B_{\{X,Y\}} \|_F^2 \big | \scriptstyle{   X_1=i_1, \ldots,X_{t} = i_{t} \atop Y_1 = j_1, \ldots, Y_{t} = j_{t}} \right ] &
    \le  \mathbb{E}\left [ \| B_{\{X,Y\}} \|_F^2 \big | \scriptstyle{   X_1=i_1, \ldots,X_{t-1} = i_{t-1} \atop Y_1 = j_1, \ldots, Y_{t-1} = j_{t-1}} \right ] \\
    & \le (k+1)^2 (\sigma_{k+1}^2 + \ldots + \sigma_m^2).
 \end{align*}
Therefore, the index sets $I$ and $J$ computed by Algorithm~\ref{alg:derandomizedACA} satisfy the bound~\eqref{eq:ACAfrob}. \end{proof}

 In analogy to the discussion in Section~\ref{sec:stoppingearly}, let us emphasize that it is not necessary to select the pair $(i_t, j_t)$ that minimizes the ratio $r$. Any pair $(i,j)$ for which the inequality
 \begin{equation}\label{eq:relaxedACA}
    (k-t+1)^2\frac{c_{m-k+t-1}(C_{ij} C_{ij}^T)}{c_{m-k+t}(C_{ij} C_{ij}^T)} \le (k+1)^2 (\sigma_{k+1}(A)^2 + \ldots + \sigma_m(A)^2)
 \end{equation}
holds will lead to index sets $I$ and $J$ such that~\eqref{eq:ACAfrob} is satisfied. Inspired by adaptive cross approximation with full pivoting~\cite{Bebendorf2000}, we traverse the entries of $B$ from the largest to the smallest (in magnitude) and stop the search once we have found an index pair $(i_t, j_t$) satisfying~\eqref{eq:relaxedACA}. 

 \subsubsection{Numerical experiments}

 We have implemented both variants of Algorithm~\ref{alg:derandomizedACA}, without and with early stopping, in Matlab. Again, the two inner loops have been implemented in a C++ function that is called via a MEX interface. The computational environment is the one described in Section~\ref{sec:numexpCS} but the test matrices 
 are smaller because Algorithm~\ref{alg:derandomizedACA} without early stopping is significantly slower. We choose $A_1$ to be $100 \times 100$, $A_2$ to be $50 \times 100$, and the matrix $A_3\in\R^{50 \times 100}$ is given by $A_3(i,j) = \left ( \left ( \frac{i}{100}\right )^{10} + \left ( \frac{j}{100}\right )^{10} \right )^{1/10}$.
 
 The left plots of Figures~\ref{fig:hilbert100ACA},~\ref{fig:exp50100ACA},~\ref{fig:p50100ACA} display the  
 approximation error $\| A - A(:,J)A(I,J)^{-1}A(I,:)\|_F$ for the index sets returned by both variants of Algorithm~\ref{alg:derandomizedACA}. It can be seen that the approximation errors often stay close to the best rank-$k$ approximation error  $\sqrt{\sigma_{k+1}^2 + \ldots + \sigma_m^2}$ and do not exceed the upper bound~\eqref{eq:ACAfrob}, modulo roundoff error.
  However, for larger values of $k$, Algorithm~\ref{alg:derandomizedACA} without early stopping appears to encounter stability issues; the approximation error is distorted well above the level of roundoff error. This appears to be due to the fact that $A(I,J)$ becomes almost singular. For instance, Algorithm~\ref{alg:derandomizedACA} without early stopping applied to $A_2$ with $k = 48$ yields a matrix $A(I,J)$ with condition number $\approx 1.3 \cdot 10^{18}$. The variant with early stopping appears to lead to lower condition numbers and does not exhibit numerical instability for the matrices considered.
 The right plots of the figures display the ratios between the execution time of Algorithm~\ref{alg:derandomizedACA} without and with early stopping, as well as the total number of index pairs that needed to be tested in Algorithm~\ref{alg:derandomizedACA} with early stopping. It can be observed that early stopping dramatically accelerates the computation and is thus the preferred variant. 
 
 \begin{figure}[ht]
    \centerline{\includegraphics[width=\textwidth]{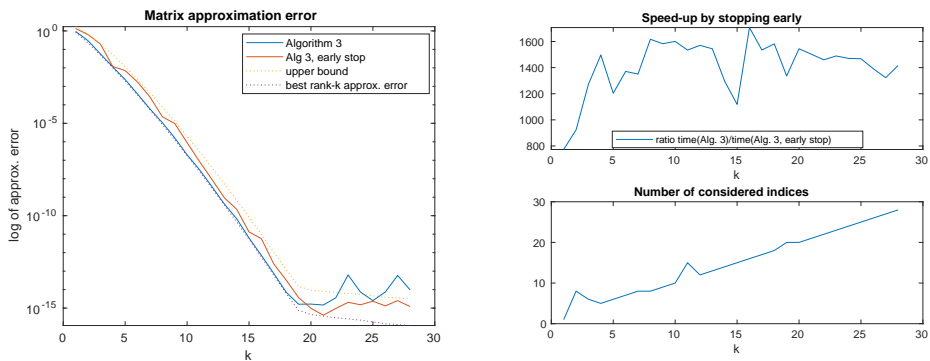}}
    \caption{Results for matrix $A_1$}
    \label{fig:hilbert100ACA}
\end{figure}

\begin{figure}[ht]
    \centerline{\includegraphics[width=\textwidth]{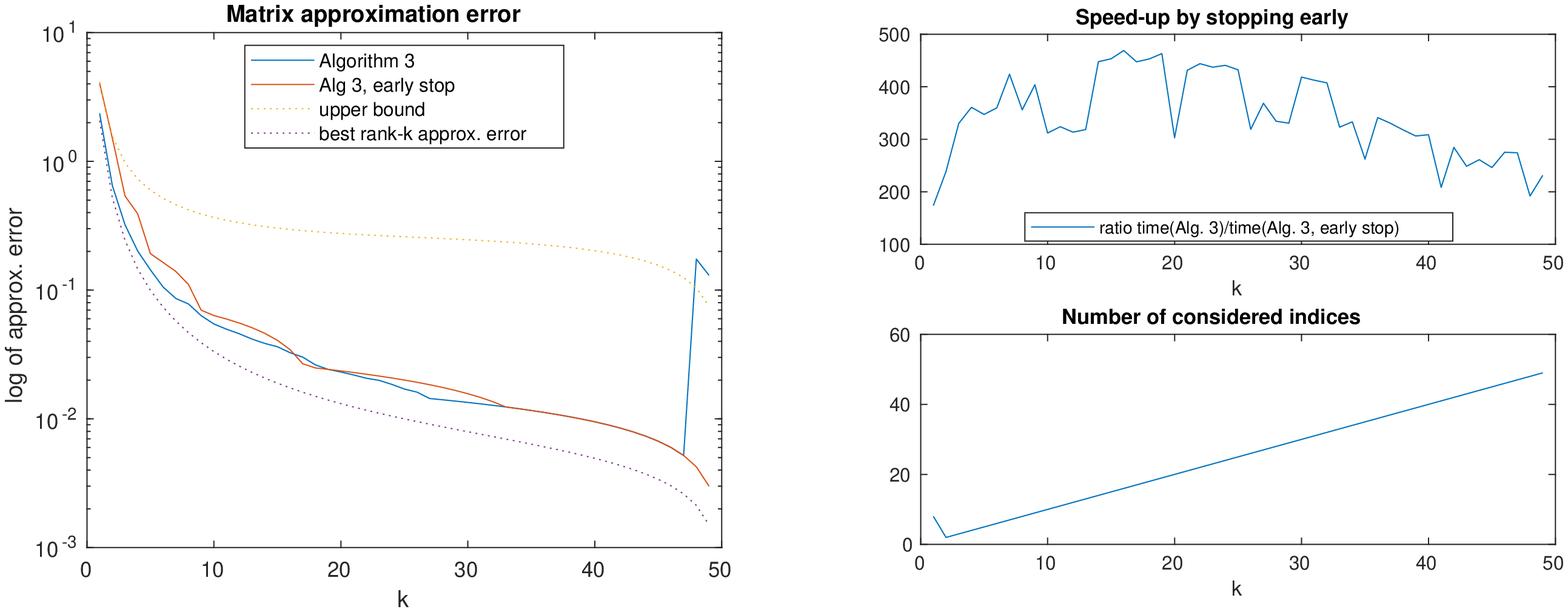}}
    \caption{Results for matrix $A_2$}
    \label{fig:exp50100ACA}
\end{figure}

\begin{figure}[ht]
    \centerline{\includegraphics[width=\textwidth]{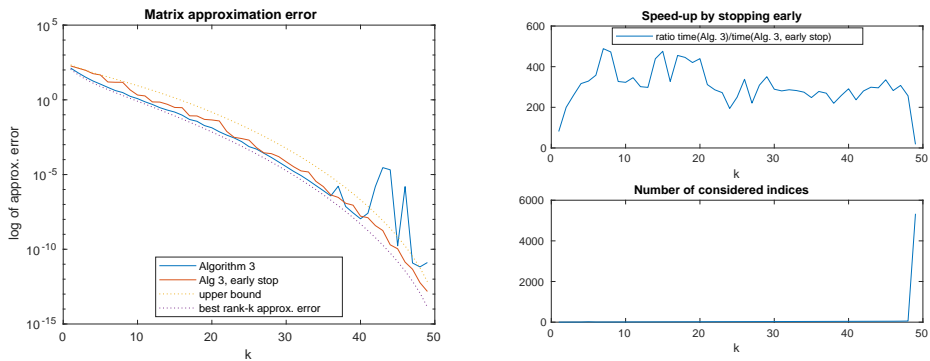}}
        \caption{Results for matrix $A_3$}
    \label{fig:p50100ACA}
\end{figure}

We also consider the $n \times n$ matrix $A = LDL^T$, where
\begin{equation*}
    L = \begin{bmatrix} 1 & & & & \\
        -c & 1 & & & \\
        -c & -c & 1 & & \\
        \vdots & \vdots & & \ddots & \\
        -c & -c & -c & \cdots & 1 \end{bmatrix}, \qquad D = \begin{bmatrix} 1 & & & &  \\ & s^2 & & & \\ & & s^4 & & \\ & & & \ddots & \\ & & & & s^{2(n-1)} \end{bmatrix}
\end{equation*}
with $s=\sin(\theta), c = \cos(\theta)$ for some $0<\theta<\pi$. This is known to be a challenging example for greedy cross approximation~\cite{Harbrecht2012}: When $k = n-1$ the greedy algorithm selects the leading $k \times k$ submatrix and returns an approximation error that is exponentially larger than the best approximation error. In contrast, Algorithm~\ref{alg:derandomizedACA}, with and without early stopping, makes the correct choice by selecting the last $n-1$ rows and columns. For instance, for $n=6$ and $\theta=0.1$, we obtain the error
\begin{equation*}
    \| A - A(:,2:6)A(2:6,2:6)^{-1} A(2:6,:)\|_F \approx 3.9 \cdot 10^{-13} < 1.8 \cdot 10^{-12} \approx \sqrt{6}\sigma_{n}.
\end{equation*}
Selecting the first $5$ rows and columns results in an error of $9.8 \cdot 10^{-11}$.

Finally, we would like to point out an interesting observation concerning the preservation of structure. In joint work with Massei~\cite{CKM2019}, we have shown that for a symmetric positive definite matrix $A$ there is always a symmetric choice of indices, $J = I$, leading to a symmetric cross approximation such that the favorable error bound of Goreinov and Tyrtyshnikov~\cite{GoreinovTyrtyshnikov2001} is attained. 
For cross approximation in the Frobenius norm, the situation appears to be more complicated; it is generally not true that a symmetric choice of indices achieves the error bound~\eqref{eq:ACAfrob} even when $A$ is symmetric positive definite. For instance, for $n=3$ and $k=1$ consider
\begin{equation*}
    A = \begin{bmatrix}
            1.87 & -1.82 & -2.11 \\
            -1.82 & 1.87 & 2.11 \\
            -2.11 & 2.11 & 2.54
        \end{bmatrix}.
\end{equation*}
The best symmetric choice is $I = J = \{ 3\}$ but this leads to an error $\approx 0.1911 > 2\sqrt{\sigma_2^2 + \sigma_3^2} \approx 0.1821$.


\section{Tensor approximation} \label{sec:tensors}

As shown, e.g., in~\cite{Drineas2007, Goreinov2008, Saibaba2016}, column subset selection can be used to approximate tensors as well. In the following, we demonstrate the use of the algorithm from Section~\ref{sec:colsel} for to obtain approximations of low multilinear rank constructed from the fibers of a third-order tensor $\mathcal{A} \in \R^{n_1 \times n_2 \times n_3}$.

First, we briefly recall some basic definitions for tensors and refer to~\cite{Kolda2009} for more details. Generalizing the notion of rows and columns of a matrix, the vectors obtained from $\mathcal{A}$ by fixing all indices but the $\mu$th one are called \emph{$\mu$-mode fibers}. The matrix $A^{(\mu)} \in \R^{n_\mu \times (n_1n_2n_3)/n_\mu}$ containing all $\mu$-mode fibers as columns is called the \emph{$\mu$-mode matricization} of $\mathcal{A}$. The $\mu$-mode product of a matrix $B \in \R^{m \times n_{\mu}}$ with $\mathcal{A}$ is denoted by $B \times_{\mu} \mathcal{A}$ and it is the tensor such that its $\mu$-mode matricization is given by $B\cdot A^{(\mu)}$.
We use the Frobenius norm of a tensor defined by
\begin{equation*}
    \| \mathcal{A} \|^2_F := {\sum_{i_1=1}^{n_1} \sum_{i_2=1}^{n_2} \sum_{i_3 = 1}^{n_3} \mathcal{A}(i_1,i_2,i_3)^2}
\end{equation*}
and recall that $\|\mathcal{A}\|_F = \| A^{(\mu)} \|_F$ for $\mu = 1,2,3$. The tuple $(k_1,k_2,k_3)$ defined by $k_\mu = \rank(A^{(\mu)})$ is called the multilinear rank of $\mathcal A$ and we can decompose $\mathcal{A}$ as
\begin{equation*}
    \mathcal{A} = B_1 \times_1 B_2 \times_2 B_3 \times_3 \mathcal{C},
\end{equation*}
for coefficient matrices $B_\mu \in \R^{n_\mu \times k_\mu}$ for $\mu=1,2,3$ and a so called \emph{core tensor} $\mathcal{C} \in \R^{k_1 \times k_2 \times k_3}$. 
This so called \emph{Tucker decomposition} is particularly beneficial when the multilinear rank is much smaller than the size of a tensor.

Algorithm~\ref{alg:CST} produces an approximate Tucker decomposition for a given tensor such each coefficient matrix $B_\mu$ is composed of $\mu$-mode fibers.
The following result shows that the obtained approximation error remains close to the best approximation error.

\begin{algorithm} \small 
    \begin{algorithmic}[1]\caption{Approximation of tensors by column selection}\label{alg:CST}
        \REQUIRE{Tensor $\mathcal{A} \in \R^{n_1 \times n_2 \times n_3}$, integers $k_1, k_2, k_3$}
        \ENSURE{Approximate Tucker decomposition of multilinear rank $(k_1,k_2,k_3)$ in terms of coefficient matrices $B_1, B_2, B_3$ and core tensor $\mathcal{C}$}
        \FOR{$\mu = 1, 2, 3$}
            \STATE{Compute $B_\mu = A^{(\mu)}(:,S_\mu)$ by applying Algorithm \ref{alg:CS} to select $k_\mu$ columns from $A^{(\mu)}$}
        \ENDFOR
        \STATE{Compute $\mathcal{C} = B_1^{+} \times_1 B_2^{+} \times_{2} B_3^{+} \times_3 \mathcal{A}$}
    \end{algorithmic}
\end{algorithm}

\begin{corollary}\label{cor:CStensors}
   Consider $\mathcal{A} \in \R^{n_1 \times n_2 \times n_3}$ and integers $k_1, k_2, k_3$ such that $1\le k_\mu\le n_\mu$ for $\mu = 1,2,3$. Then the output of Algorithm~\ref{alg:CST} satisfies
    \begin{equation*}
        \| \mathcal{A} - B_1 \times_1 B_2 \times_2 B_3 \times_3 \mathcal{C} \|_F \le \sqrt{k_1 + k_2 + k_3 + 3} \cdot \| \mathcal{A} - \mathcal{A}_{\mathrm{best}} \|_F,
    \end{equation*}
    where $\mathcal{A}_{\mathrm{best}}$ is the best Tucker approximation of $\mathcal{A}$ of multilinear rank at most $(k_1, k_2, k_3)$.
\end{corollary}
\begin{proof}
The proof is similar to existing proofs on the quasi-optimality of the Higher-Order SVD~\cite{DeLathauwer2000} and related results in~\cite{Drineas2007, Goreinov2008, Saibaba2016}.

Using~\eqref{eq:CS} and setting $\pi_{\mu} = B_\mu B_{\mu}^{+}$, the result of Algorithm~\ref{alg:CS} applied to $A^{(\mu)}$ satisfies
    \begin{align*}
        \| A^{(\mu)} - \pi_{\mu}(A^{(\mu)}) \|_F^2 & \le (k_\mu + 1) \big( \sigma_{k_\mu + 1}( A^{(\mu)} )^2 + \cdots  + \sigma_{n_\mu}( A^{(\mu)} )^2 \big)  \\
        & \le (k_\mu + 1) \| A^{(\mu)} - A_{\mathrm{best}}^{(\mu)} \|_F^2 
        = (k_\mu + 1) \| \mathcal{A} - \mathcal{A}_{\mathrm{best}} \|_F^2,
    \end{align*}
    where the second inequality follows from the fact that $A_{\mathrm{best}}^{(\mu)}$, the $\mu$-mode matricization of $\mathcal{A}_{\mathrm{best}}$, has rank at most $k_\mu$. Using the orthogonality of the projections $\pi_\mu$, we obtain
    \begin{align*}
        \| \mathcal{A} & - B_1 \times B_2 \times B_3 \times \mathcal{C} \|_F^2 =
        \| \mathcal{A}  - \pi_1 \times \pi_2 \times \pi_3 \times \mathcal{A} \|_F^2 \\
        & = \| \mathcal{A} - \pi_1 \times \mathcal{A} \|_F^2 + \| \pi_1 \times (\mathcal{A} - \pi_2 \times \mathcal{A}) \|_F^2 + \| \pi_1 \times \pi_2 \times (\mathcal{A} - \pi_3 \times \mathcal{A}) \|_F^2 \\
        & \le \sum_{\mu = 1}^3 \| \mathcal{A} - \pi_\mu \times \mathcal{A} \|_F^2 = \sum_{\mu = 1}^3 \| A^{(\mu)} - \pi_{\mu}(A^{(\mu)}) \|_F^2 \le \sum_{\mu = 1}^3 (k_\mu + 1) \|\mathcal{A} - \mathcal{A}_{\mathrm{best}} \|_F^2 \\
        & = (k_1 + k_2 + k_3 + 3)\| \mathcal{A} - \mathcal{A}_{\mathrm{best}} \|_F^2,
    \end{align*}
    where the second equality follows from~\cite[Theorem 5.1]{Vannieuwenhoven2012}.
\end{proof}

\begin{remark} Algorithm~\ref{alg:CST} easily generalizes to tensors of arbitrary order. Given a tensor $\mathcal{A}\in\R^{n_1 \times \ldots \times n_d}$ and integers $k_1, \ldots, k_d$, this generalization constructs subsets of fibers $B_1, \ldots, B_d$ and a core tensor $\mathcal{C}$ such that
\begin{equation*}
    \| \mathcal{A} - B_1 \times_1 \ldots \times_{d-1} B_d \times_d \mathcal{C} \|_F \le \sqrt{k_1 + \ldots + k_d + d} \cdot \| \mathcal{A} - \mathcal{A}_{\mathrm{best}} \|_F.
\end{equation*}
\end{remark}

\subsection{Numerical experiments}
We have implemented Algorithm~\ref{alg:CST} in Matlab and tested it on two $50 \times 50 \times 50$ tensors, given by $\mathcal A_1(i,j,h) = \frac{1}{i+j+h-1}$ and $\mathcal A_2(i,j,h) = \left ( i^{10} + j^{10} + h^{10} \right )^{1/10} / 50$. We choose $k_1 = k_2 = k_3 = k$ and report in Figure~\ref{fig:plotsTensor} the obtained approximation errors $\| \mathcal{A}_i - B_1 \times B_2 \times B_3 \times \mathcal{C} \|_F$ for different values of $k$, where $B_1 ,B_2, B_3$, $\mathcal{C}$ are returned by Algorithm~\ref{alg:CST}, with and without early stopping in the column selection part. We compare with the quantity \[
 \Big( \sum_{\mu = 1}^3 \sigma_{k_\mu + 1}( A^{(\mu)} )^2 + \cdots  + \sigma_{n_\mu}( A^{(\mu)} )^2 \Big)^{1/2},
\]
which provides a (tight) upper bound on the best approximation error.
It can be seen that the errors obtained from Algorithm~\ref{alg:CST} remain close to this quasi-best approximation error.
 \begin{figure}[ht]
    \centerline{\includegraphics[width=\textwidth]{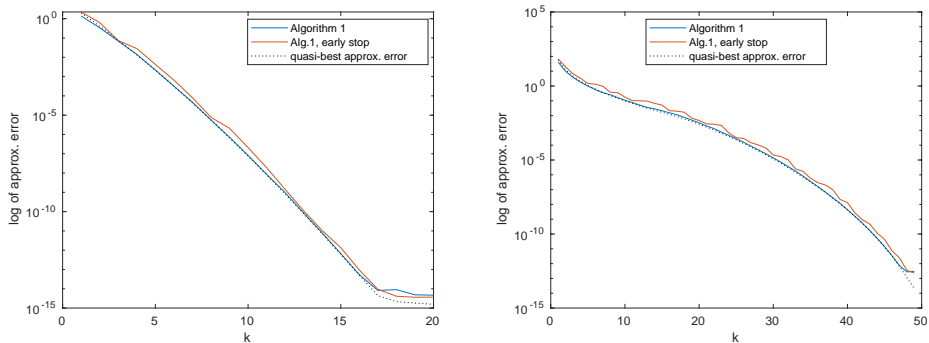}}
    \caption{Results for tensors $\mathcal{A}_1$ (left) and $\mathcal{A}_2$ (right).    \label{fig:plotsTensor}}
\end{figure}

\section{Conclusions}

In this work, we have proposed several improvements to the column selection algorithm by Deshpande and Rademacher~\cite{Deshpande2010}. The numerical experiments indicate that updating singular values (instead of characteristic polynomials) leads to numerical robustness, in the sense that the approximation error obtained in finite precision arithmetic is not affected unduly by roundoff error. We have also developed an extension of~\cite{Deshpande2010} to produce cross approximations of matrices and, to the best of our knowledge, this extension constitutes the first deterministic polynomial time algorithm that yields a cross approximation with a guaranteed polynomial error bound. 
We have introduced a mechanism for stopping early the search for indices in column subset selection or cross approximation. Although relatively simple, this mechanism tremendously reduces the execution time for all examples tested.

A number of issues remain for future study, such as the numerical stability analysis of our algorithms. In particular, it would be desirable to study the numerical robustness of the cross approximation returned by Algorithm~\ref{alg:derandomizedACA} with early stopping. Also, by combining early stopping with a more aggressive reuse of the SVD might lead to further complexity reduction, but a rigorous complexity analysis would require deeper understanding of early stopping, well beyond the limited scope of Lemma~\ref{lemma:indication}. Finally, we would like to stress that the algorithms presented in this work are intented for small to medium sized matrices and tensors. For large-scale data, the 
algorithms presented in this paper need to be combined with other, possibly heuristic dimensionality reduction techniques.

\begin{paragraph}{Acknowledgements.} The authors thank Sergey Dolgov for helpful discussions on topics related to the work presented in this paper.
\end{paragraph}

	\bibliographystyle{abbrv}
	\bibliography{Bib} 

\begin{thebibliography}{10}

\bibitem{Anderson1999}
E.~Anderson, Z.~Bai, C.~H. Bischof, S.~Blackford, J.~W. Demmel, J.~J. Dongarra,
  J.~Du~Croz, A.~Greenbaum, S.~Hammarling, A.~McKenney, and D.~C. Sorensen.
\newblock {\em {LAPACK} Users' Guide}.
\newblock SIAM, Philadelphia, PA, third edition, 1999.

\bibitem{Aurentz2018}
J.~L. Aurentz, T.~Mach, L.~Robol, R.~Vandebril, and D.~S. Watkins.
\newblock {\em Core-chasing algorithms for the eigenvalue problem}, volume~13
  of {\em Fundamentals of Algorithms}.
\newblock SIAM, Philadelphia, PA, 2018.

\bibitem{Barrault2004}
M.~Barrault, Y.~Maday, N.~C. Nguyen, and A.~T. Patera.
\newblock An `empirical interpolation' method: application to efficient
  reduced-basis discretization of partial differential equations.
\newblock {\em C. R. Math. Acad. Sci. Paris}, 339(9):667--672, 2004.

\bibitem{Bebendorf2000}
M.~Bebendorf.
\newblock Approximation of boundary element matrices.
\newblock {\em Numer. Math.}, 86(4):565--589, 2000.

\bibitem{Bernstein2009}
D.~S. Bernstein.
\newblock {\em Matrix Mathematics}.
\newblock Princeton University Press, Princeton, NJ, second edition, 2009.

\bibitem{Boutsidis2009}
C.~Boutsidis, M.~W. Mahoney, and P.~Drineas.
\newblock An improved approximation algorithm for the column subset selection
  problem.
\newblock In {\em Proceedings of the {T}wentieth {A}nnual {ACM}-{SIAM}
  {S}ymposium on {D}iscrete {A}lgorithms}, pages 968--977. SIAM, Philadelphia,
  PA, 2009.

\bibitem{Buergisser1997}
P.~B\"{u}rgisser, M.~Clausen, and M.~A. Shokrollahi.
\newblock {\em Algebraic complexity theory}, volume 315 of {\em Grundlehren der
  Mathematischen Wissenschaften}.
\newblock Springer-Verlag, Berlin, 1997.

\bibitem{Civril2009}
A.~\c{C}ivril and M.~Magdon-Ismail.
\newblock On selecting a maximum volume sub-matrix of a matrix and related
  problems.
\newblock {\em Theoret. Comput. Sci.}, 410(47-49):4801--4811, 2009.

\bibitem{Chan1987}
T.~F. Chan.
\newblock Rank revealing {$QR$} factorizations.
\newblock {\em Linear Algebra Appl.}, 88/89:67--82, 1987.

\bibitem{Chandrasekaran1994}
S.~Chandrasekaran and I.~C.~F. Ipsen.
\newblock On rank-revealing factorisations.
\newblock {\em SIAM J. Matrix Anal. Appl.}, 15(2):592--622, 1994.

\bibitem{Chaturantabut2010}
S.~Chaturantabut and D.~C. Sorensen.
\newblock Nonlinear model reduction via discrete empirical interpolation.
\newblock {\em SIAM J. Sci. Comput.}, 32(5):2737--2764, 2010.

\bibitem{CKM2019}
A.~Cortinovis, D.~Kressner, and S.~Massei.
\newblock On maximum volume submatrices and cross approximation for symmetric
  semidefinite and diagonally dominant matrices.
\newblock {\em arXiv preprint arXiv:1902.02283}, 2019.

\bibitem{DeLathauwer2000}
L.~De~Lathauwer, B.~De~Moor, and J.~Vandewalle.
\newblock A multilinear singular value decomposition.
\newblock {\em SIAM J. Matrix Anal. Appl.}, 21(4):1253--1278, 2000.

\bibitem{Deshpande2010}
A.~Deshpande and L.~Rademacher.
\newblock Efficient volume sampling for row/column subset selection.
\newblock In {\em 2010 {IEEE} 51st {A}nnual {S}ymposium on {F}oundations of
  {C}omputer {S}cience---{FOCS} 2010}, pages 329--338. IEEE Computer Soc., Los
  Alamitos, CA, 2010.

\bibitem{Deshpande2006}
A.~Deshpande, L.~Rademacher, S.~Vempala, and G.~Wang.
\newblock Matrix approximation and projective clustering via volume sampling.
\newblock {\em Theory Comput.}, 2:225--247, 2006.

\bibitem{Drineas2007}
P.~Drineas and M.~W. Mahoney.
\newblock A randomized algorithm for a tensor-based generalization of the
  singular value decomposition.
\newblock {\em Linear Algebra Appl.}, 420(2-3):553--571, 2007.

\bibitem{Drineas2008}
P.~Drineas, M.~W. Mahoney, and S.~Muthukrishnan.
\newblock Relative-error {$CUR$} matrix decompositions.
\newblock {\em SIAM J. Matrix Anal. Appl.}, 30(2):844--881, 2008.

\bibitem{Drmac2016}
Z.~Drma\v{c} and S.~Gugercin.
\newblock A new selection operator for the discrete empirical interpolation
  method---improved a priori error bound and extensions.
\newblock {\em SIAM J. Sci. Comput.}, 38(2):A631--A648, 2016.

\bibitem{Frieze2004}
A.~Frieze, R.~Kannan, and S.~Vempala.
\newblock Fast {M}onte-{C}arlo algorithms for finding low-rank approximations.
\newblock {\em J. ACM}, 51(6):1025--1041, 2004.

\bibitem{GolubVanLoan2013}
G.~H. Golub and C.~F. Van~Loan.
\newblock {\em Matrix computations}.
\newblock Johns Hopkins Studies in the Mathematical Sciences. Johns Hopkins
  University Press, Baltimore, MD, fourth edition, 2013.

\bibitem{GoreinovTyrtyshnikov2001}
S.~A. Goreinov and E.~E. Tyrtyshnikov.
\newblock The maximal-volume concept in approximation by low-rank matrices.
\newblock In {\em Structured matrices in mathematics, computer science, and
  engineering, {I} ({B}oulder, {CO}, 1999)}, volume 280 of {\em Contemp.
  Math.}, pages 47--51. Amer. Math. Soc., Providence, RI, 2001.

\bibitem{Goreinov2008}
S.~A. Gore\u{\i}nov.
\newblock Cross approximation of a multi-index array.
\newblock {\em Dokl. Akad. Nauk}, 420(4):439--441, 2008.

\bibitem{Gu1996}
M.~Gu and S.~C. Eisenstat.
\newblock Efficient algorithms for computing a strong rank-revealing {QR}
  factorization.
\newblock {\em SIAM J. Sci. Comput.}, 17(4):848--869, 1996.

\bibitem{Harbrecht2012}
H.~Harbrecht, M.~Peters, and R.~Schneider.
\newblock On the low-rank approximation by the pivoted {C}holesky
  decomposition.
\newblock {\em Appl. Numer. Math.}, 62(4):428--440, 2012.

\bibitem{HornJohnson2013}
R.~A. Horn and C.~R. Johnson.
\newblock {\em Matrix analysis}.
\newblock Cambridge University Press, Cambridge, second edition, 2013.

\bibitem{Knill2014}
O.~Knill.
\newblock Cauchy-{B}inet for pseudo-determinants.
\newblock {\em Linear Algebra Appl.}, 459:522--547, 2014.

\bibitem{Kolda2009}
T.~G. Kolda and B.~W. Bader.
\newblock Tensor decompositions and applications.
\newblock {\em SIAM Rev.}, 51(3):455--500, 2009.

\bibitem{RehmanIpsen2011b}
R.~Rehman and I.~C. Ipsen.
\newblock La budde's method for computing characteristic polynomials.
\newblock {\em arXiv preprint arXiv:1104.3769}, 2011.

\bibitem{RehmanIpsen2011}
R.~Rehman and I.~C.~F. Ipsen.
\newblock Computing characteristic polynomials from eigenvalues.
\newblock {\em SIAM J. Matrix Anal. Appl.}, 32(1):90--114, 2011.

\bibitem{Saibaba2016}
A.~K. Saibaba.
\newblock H{OID}: higher order interpolatory decomposition for tensors based on
  {T}ucker representation.
\newblock {\em SIAM J. Matrix Anal. Appl.}, 37(3):1223--1249, 2016.

\bibitem{Sorensen2016}
D.~C. Sorensen and M.~Embree.
\newblock A {DEIM} induced {CUR} factorization.
\newblock {\em SIAM J. Sci. Comput.}, 38(3):A1454--A1482, 2016.

\bibitem{SorensenEmbree2016}
D.~C. Sorensen and M.~Embree.
\newblock A {DEIM} induced {CUR} factorization.
\newblock {\em SIAM J. Sci. Comput.}, 38(3):A1454--A1482, 2016.

\bibitem{Stewart1999}
G.~W. Stewart.
\newblock Four algorithms for the efficient computation of truncated pivoted
  {QR} approximations to a sparse matrix.
\newblock {\em Numer. Math.}, 83(2):313--323, 1999.

\bibitem{Townsend2013}
A.~Townsend and L.~N. Trefethen.
\newblock An extension of {C}hebfun to two dimensions.
\newblock {\em SIAM J. Sci. Comput.}, 35(6):C495--C518, 2013.

\bibitem{Vannieuwenhoven2012}
N.~Vannieuwenhoven, R.~Vandebril, and K.~Meerbergen.
\newblock A new truncation strategy for the higher-order singular value
  decomposition.
\newblock {\em SIAM J. Sci. Comput.}, 34(2):A1027--A1052, 2012.

\bibitem{Woodruff2014}
D.~P. Woodruff.
\newblock Sketching as a tool for numerical linear algebra.
\newblock {\em Found. Trends Theor. Comput. Sci.}, 10(1-2):iv+157, 2014.

\bibitem{Yoon1996}
P.~A.-C. Yoon.
\newblock {\em Modifying two-sided orthogonal decompositions: {A}lgorithms,
  implementation, and applications}.
\newblock ProQuest LLC, Ann Arbor, MI, 1996.
\newblock Thesis (Ph.D.)--The Pennsylvania State University.

\bibitem{Zamarashkin2018}
N.~L. Zamarashkin and A.~I. Osinsky.
\newblock On the existence of a nearly optimal skeleton approximation of a
  matrix in the {F}robenius norm.
\newblock In {\em Doklady Mathematics}, volume~97, pages 164--166. Springer,
  2018.

\end{thebibliography}
	
	\end{document}